\documentclass[10pt]{amsart}
\usepackage[margin=2.7cm]{geometry}
\usepackage{amsmath,amssymb,amsthm,mathtools}
\usepackage{booktabs,array}
\usepackage{url}
\usepackage{enumitem}
\usepackage{xcolor}
\usepackage{listings}
\usepackage[expansion=false]{microtype}
\usepackage[colorlinks=true,linkcolor=blue!50!black,citecolor=blue!50!black,urlcolor=blue!50!black]{hyperref}

\usepackage{embedfile}
\ifdefined\pdfgentounicode
  \input{glyphtounicode}
\fi
\newtheorem{theorem}{Theorem}[section]
\newtheorem{proposition}[theorem]{Proposition}
\newtheorem{lemma}[theorem]{Lemma}
\newtheorem{corollary}[theorem]{Corollary}
\theoremstyle{definition}

\newtheorem{remark}[theorem]{Remark}

\newtheorem{question}[theorem]{Question}

\newcommand{\PP}{\mathbb{P}}
\newcommand{\CC}{\mathbb{C}}
\newcommand{\eorb}{e_{\rm orb}}
\newcommand{\lct}{{\rm lct}}
\newcommand{\mult}{{\rm mult}}
\DeclareMathOperator{\Sing}{Sing}

\begin{document}

\title{There are no Miyaoka--Kobayashi octics}

\author{Alexandru Dimca}
\address{Alexandru Dimca: Universit\'e C\^ote d'Azur, CNRS, LJAD, France}
\address{Simion Stoilow Institute of Mathematics, Bucharest, Romania}
\email{Alexandru.DIMCA@univ-cotedazur.fr}

\author{Piotr Pokora}
\address{Piotr Pokora: Department of Mathematics, University of the National Education Commission Krakow,
Podchor\c a\.zych 2, PL-30-084 Krak\'ow, Poland}
\email{piotr.pokora@up.krakow.pl}
\dedicatory{To Wolfgang Ebeling, in memoriam}

\begin{abstract}
Let $C$ be a reduced plane curve of degree $n$ with only ADE singularities. Using K\"ahler--Einstein metrics, Kobayashi proved an upper bound, in terms of $n$, for a sum of local invariants of the singular points of $C$. For even $n$ this bound had been obtained earlier by Hirzebruch and Ivinskis. Curves of even degree for which equality holds are called Miyaoka--Kobayashi curves, and the only known examples are three types of plane sextics. We show that there are no Miyaoka--Kobayashi curves of degree $8$, and additionally that a Miyaoka--Kobayashi curve of degree $10$ cannot be free. The key new ingredient is a new linear constraint on the singularities of such curves, obtained by differentiating Langer's orbifold Bogomolov--Miyaoka--Yau inequality.
\end{abstract}

\maketitle

\section{Introduction}

Let $C\subset\PP^2$ be a reduced complex plane curve of degree $n$ having only ADE singularities. For a singular point $p$ of $C$ we denote by $\mu_p=\tau_p$ its Milnor (equal to its Tjurina) number, and by $G_p\subset SU(2)$ the binary polyhedral group of the ADE surface singularity $z^2=f_p(x,y)$ with $f_p=0$ being a local equation of $C$ at $p$. Using K\"ahler--Einstein metrics on log canonical surfaces with branch loci, Kobayashi \cite[Corollary 2]{Ko} proved that for $n\geq7$ the following inequality holds:
\begin{equation}\label{eq:Ko}
\sum_{p\in\Sing(C)}\Bigl(\mu_p+1-\frac1{|G_p|}\Bigr)\leq\frac56n^2-n.
\end{equation}
Note that $|\Gamma(p)|=2|G_p|$ in the notation of \cite{Ko}. For $n$ even, \eqref{eq:Ko} was obtained before by Hirzebruch and Ivinskis \cite{Hi,Iv} using the double cover of $\PP^2$ branched along $C$. With the Miyaoka--Kobayashi numbers defined as $m(p)=3(\mu_p+1-1/|G_p|)$, see \cite[Lemma 3.1]{DP}, the inequality \eqref{eq:Ko} reads as
$$m(C):=\sum_{p\in\Sing(C)}m(p)\leq\frac n2(5n-6),$$
and a curve $C$ of even degree $n\geq6$ is called a \emph{Miyaoka--Kobayashi curve}, or an \emph{MK-curve}, when equality holds, see \cite{Hi, Iv} and \cite[Definition 2.4]{DP}. For $n\geq8$, the double cover of $\PP^2$ branched along an MK-curve is a ball quotient surface, see \cite[Section 2]{DP} for necessary details. Ivinskis proved in his master’s thesis \cite{Iv} that there are exactly three types of MK-curves of degree $6$ (see also \cite[Theorem 3.9]{DP}). By contrast, \emph{no} MK-curve of degree $n \geq 8$ is known. In particular, the existence of MK-octics has remained open since Ivinskis raised the problem in his Bonn Diplomarbeit \cite{Iv}, supervised by Hirzebruch.

The inequality \eqref{eq:Ko} is the case $D=\frac12C$ of the orbifold Miyaoka--Yau inequality for the pair $(\PP^2,D)$, see \cite[Theorem]{Ko}. Langer \cite{La} proved such an inequality for the pairs $(\PP^2,\alpha C)$ with any real $\alpha$ for which the pair is log canonical and $K_{\PP^2}+\alpha C$ is pseudoeffective, see Theorem \ref{thm:Langer} below. This inequality was the main tool used in \cite[Section 5]{DP}. For $\alpha=\frac12$ it gives back \eqref{eq:Ko}, so an MK-curve gives an equality in Langer's inequality at $\alpha=\frac12$. When $n\geq8$, the value $\frac12$ is an interior point of the interval of admissible values of $\alpha$, and both sides of Langer's inequality depend smoothly on $\alpha$ near $\frac12$. Hence the derivatives of both sides at $\alpha=\frac12$ must be equal, and this gives a \textbf{linear constraint} on the singularities of an MK-curve. Recall that $\tau(C)=\sum_p\tau_p$ is the total Tjurina number of $C$.

\begin{theorem}\label{thm:main}
Let $C$ be an MK-curve of even degree $n\geq 8$. Then
$$\sum_{p\in\Sing(C)}c_p=2n^2-3n,$$
where $c_p=3k$ if $p$ is a point of type $A_k$, $c_p=\frac{3(k-1)(2k-3)}{2(k-2)}$ if $p$ is of type $D_k$, and $c_p=\frac{33}{2},\frac{153}{8},\frac{87}{4}$ if $p$ is of type $E_6,E_7,E_8$ respectively. Equivalently, one has
$$3\tau(C)-\sum_{p\in\Sing(C)}d_p=2n^2-3n,$$
where $d_p=3\tau_p-c_p$, i.e. $d_p=0$ for $A_k$, $d_p=\frac{3(k-3)}{2(k-2)}$ for $D_k$, and $d_p=\frac32,\frac{15}{8},\frac94$ for $E_6,E_7,E_8$ respectively.
\end{theorem}

Note that $d_p\geq\frac34$ for any singularity of type $D$ or $E$. 

Let us recall that the \emph{freeness defect} $\nu(C)$ of $C$, which was introduced by Dimca in \cite{Di1}, is $0$ (resp. $1$) exactly when $C$ is free (resp. nearly free), and that $\nu(C)=3m^2-3m+1-\tau(C)$ for any curve $C$ of degree $n=2m$ with only ADE singularities, see \cite[Section 6]{DP}. The following consequences are immediate.

\begin{corollary}\label{cor:main}
Let $C$ be an MK-curve of degree $n=2m\geq 8$. Then the following hold.
\begin{enumerate}[label=(\roman*)]
\item $\tau(C)\geq\frac{8m^2-6m}{3}$, or equivalently $\nu(C)\leq\frac{m^2-3m+3}{3}$.
\item If $n$ is not divisible by $3$, then $C$ has at least one singularity of type $D$ or $E$. In particular, $C$ cannot be nodal or cuspidal.
\item The number of singularities of type $D$ or $E$ on $C$ is at most $\frac43(m^2-3m+3)$.
\end{enumerate}
In particular, an MK-curve of degree $8$ satisfies $35\leq\tau(C)\leq 37$ and $\nu(C)\leq 2$, while an MK-curve of degree $10$ satisfies $57\leq\tau(C)\leq 61$ and $\nu(C)\leq 4$.
\end{corollary}

Theorem \ref{thm:main}, the MK condition and the upper bound $\tau(C)\leq 3m(m-1)+1$ reduce the possible singularities of an MK-curve of degree $8$ to $18$ weak combinatorial types, see Proposition \ref{prop:octics}. A case by case analysis, based only on many beautiful classical facts on plane curves (the genus formula, B\'ezout's theorem, the Pl\"ucker formulas and projective duality), shows that none of them can be realized. Hence we get our main somehow surprising result.

\begin{theorem}\label{thm:octics}
There are no Miyaoka--Kobayashi curves of degree $8$.
\end{theorem}

In degree $10$ the same approach leads us to $284$ numerical types, see Proposition \ref{prop:decics}. This is way too much for a case by case analysis and we do not know how to approach this problem effectively, but we can at least prove the following result.

\begin{theorem}\label{thm:decics}
Let $C$ be an MK-curve of degree $10$. Then $57\leq\tau(C)\leq 60$. In particular, $C$ is not a free curve and $1\leq\nu(C)\leq 4$. Moreover, $C$ has at least one singularity of type $D$ or $E$.
\end{theorem}

Our proofs are based on small amount of symbolic computations, just at two places, namely for the enumeration of the numerical types in Propositions \ref{prop:octics} and \ref{prop:decics}. In Appendix \ref{app:scripts} we give simple \texttt{SINGULAR} scripts performing these enumerations, together with a script recomputing the numerical data in Table \ref{tab:octics}. All the other arguments are done by hand.

Let us present the plan for our paper. In Section \ref{sec:main} we explain how \eqref{eq:Ko} fits into Langer's inequality and prove Theorem \ref{thm:main} and Corollary \ref{cor:main}. Section \ref{sec:facts} collects the classical facts on plane curves that we need, and Sections \ref{sec:octics} and \ref{sec:decics} contain the proofs of Theorems \ref{thm:octics} and \ref{thm:decics}. Appendix \ref{app:eorb} collects the local orbifold Euler numbers of the ADE singularities, and Appendix \ref{app:scripts} contains the computer scripts mentioned above.

We work only in the projective setting over the complex numbers.

\section{Kobayashi's inequality and Langer's inequality}\label{sec:main}

Any ADE curve singularity $p$ has a weighted homogeneous normal form $f$, of weighted degree $1$ with respect to weights $w_1\geq w_2$, see Table \ref{tab:local}. Its log canonical threshold is 
$$\lct_p=\min(1,w_1+w_2)>\frac{1}{2},$$ see \cite[Proposition~2.3]{DimSer}.

\begin{table}[ht]
\centering\small
\renewcommand{\arraystretch}{1.4}
\begin{tabular}{@{}lllllll@{}}
\toprule
type & normal form & $\mu_p$ & $|G_p|$ & $(w_1,w_2)$ & $\alpha_p$ & $c_p$\\
\midrule
$A_k$, $k\geq1$ & $x^2+y^{k+1}$ & $k$ & $k+1$ & $(\frac12,\frac1{k+1})$ & $\frac12-\frac1{k+1}$ & $3k$\\
$D_k$, $k\geq4$ & $x^2y+y^{k-1}$ & $k$ & $4(k-2)$ & $\bigg(\frac{k-2}{2(k-1)},\frac1{k-1}\bigg)$ & $\frac{k-4}{2(k-3)}$ & $\frac{3(k-1)(2k-3)}{2(k-2)}$\\
$E_6$ & $x^3+y^4$ & $6$ & $24$ & $(\frac13,\frac14)$ & $\frac1{12}$ & $\frac{33}{2}$\\
$E_7$ & $x^3+xy^3$ & $7$ & $48$ & $(\frac13,\frac29)$ & $\frac19$ & $\frac{153}{8}$\\
$E_8$ & $x^3+y^5$ & $8$ & $120$ & $(\frac13,\frac15)$ & $\frac2{15}$ & $\frac{87}{4}$\\
\bottomrule
\end{tabular}
\caption{Local data of ADE singularities.}
\label{tab:local}
\end{table}

We use Langer's inequality \cite[Corollary 5.2]{La} for $X=\PP^2$ and the $\mathbb{Q}$-divisor $\alpha C$, in the form \cite[(11.1.1)]{La}, see also \cite[Theorem 5.3]{DP}. Note that $K_{\PP^2}+\alpha C$ is pseudoeffective exactly when $\alpha\geq\frac3n$.

\begin{theorem}[Langer]\label{thm:Langer}
Let $C\subset\PP^2$ be a reduced curve of degree $n$ and let $\alpha\in[\frac3n,1]$ be such that the pair $(\PP^2,\alpha C)$ is log canonical, i.e.\ $\alpha\leq\lct_{p}(C)$ for all $p\in\Sing(C)$. Then
$$\sum_{p\in\Sing(C)}3\Bigl(\alpha(\mu_p-1)+1-\eorb(p,\alpha)\Bigr)\leq(3\alpha-\alpha^2)n^2-3\alpha n,$$
where $\eorb(p,\alpha)=\eorb(p;\PP^2,\alpha C)$ is the local orbifold Euler number of the pair $(\PP^2,\alpha C)$ at $p$.
\end{theorem}

For the ADE singularities we only need values of $\eorb(p,\alpha)$ for $\alpha$ being close to $\frac12$. The values for all $0\leq\alpha\leq\lct_p$ are collected in Table \ref{tab:eorb} in Appendix \ref{app:eorb}.

\begin{lemma}\label{lem:eorb}
Let $p$ be an ADE singularity with weights $(w_1,w_2)$. Then
$$\eorb(p,\alpha)=\frac{(w_1+w_2-\alpha)^2}{4w_1w_2}\quad\text{ for }\ \alpha_p\leq\alpha\leq\lct_p,$$
where $\alpha_p<\frac12$ is given in Table \ref{tab:local}. In particular, $\eorb(p,\frac12)=\frac1{2|G_p|}$.
\end{lemma}

\begin{proof}
This is \cite[Theorem 9.4.2]{La} applied to the minimal log resolution of $(\CC^2,C)$, which is star-shaped for all ADE singularities, see also \cite{Bo}. For $A_1$ and $D_4$ it also follows from \cite[Theorem 8.3]{La}, and the cusp $A_2$ is treated in \cite[11.2]{La}. The value at $\alpha=\frac12$ is obtained by substituting the weights from Table \ref{tab:local}. It can also be seen directly: the double cover $Y$ of $\CC^2$ branched along $C$ satisfies $K_Y=\pi^*(K_{\CC^2}+\frac12C)$ and has the quotient singularity $\CC^2/G_p$ over $p$, so \cite[(0.2.2) and Lemma 7.1]{La} give $\eorb(p,\frac12)=\frac12\cdot\frac1{|G_p|}$.
\end{proof}

Let $\Phi(\alpha)$ be the right-hand side minus the left-hand side of the inequality in Theorem \ref{thm:Langer}, so that $\Phi(\alpha)\geq0$ for $\frac3n\leq\alpha\leq\min_p\lct_p$. For $\alpha=\frac12$, Lemma \ref{lem:eorb} gives
$$\Phi\bigl(\tfrac12\bigr)=\frac54n^2-\frac32n-\frac32\sum_{p\in\Sing(C)}\Bigl(\mu_p+1-\frac1{|G_p|}\Bigr).$$
Hence, for $n\geq6$, Langer's inequality for $\alpha=\frac12$ is exactly Kobayashi's inequality \eqref{eq:Ko}, and a curve $C$ of even degree is an MK-curve if and only if $\Phi(\frac12)=0$.

\begin{proof}[Proof of Theorem \ref{thm:main}]
Since $n\geq8$, we have $\frac3n<\frac12<\min_p\lct_p$, and by Lemma \ref{lem:eorb} the function $\Phi$ is a polynomial in $\alpha$ on a neighbourhood of $\frac12$. As $\Phi\geq0$ and $\Phi(\frac12)=0$, we get $\Phi'(\frac12)=0$. The derivative of the right-hand side of Langer's inequality at $\frac12$ is $2n^2-3n$, and by Lemma \ref{lem:eorb} the derivative of the term corresponding to $p$ on the left-hand side is
$$3(\mu_p-1)+\frac{3(w_1+w_2-\frac12)}{2w_1w_2}=c_p,$$
with the values given in Table \ref{tab:local}.
\end{proof}

\begin{proof}[Proof of Corollary \ref{cor:main}]
Since $d_p\geq0$, Theorem \ref{thm:main} gives $3\tau(C)\geq 2n^2-3n=8m^2-6m$, which is (i). If all the singularities of $C$ are of type $A$, then $d_p=0$ for all $p$, and $3\tau(C)=8m^2-6m$ implies that $3$ divides $m$, hence also $n$. This proves (ii). For (iii), recall that $\tau(C)\leq3m(m-1)+1$, see the proof of \cite[Theorem 3.4]{DP}, which relies on \cite{dPW} and \cite[Theorem 2.9]{DP2}. Hence Theorem \ref{thm:main} gives $\sum_pd_p\leq m^2-3m+3$, and the claim follows from $d_p\geq\frac34$ for the points of type $D$ and $E$.
\end{proof}

\begin{remark}\label{rem:checks}
\begin{enumerate}[label=(\arabic*)]
\item For $n=6$ one has $\frac3n=\frac12$, so the argument above only gives $\Phi'(\frac12)\geq0$, i.e. $\sum_pc_p\leq54$. Nevertheless, the three types of MK-sextics described in \cite[Definition 3.8]{DP} all satisfy the condition $\sum_pc_p=54=2n^2-3n$:
$$6c(A_1)+4c(A_3)=18+36,\qquad 9c(A_2)=54,\qquad 3c(A_1)+4c(D_4)=9+45.$$
\item In \cite[Section 5]{DP}, the cases (b) and (c) for octics with singularities of types $A_1,A_2,A_3,D_4$ were excluded using $\alpha=0.55$ and $\alpha=0.48$, respectively. This choice is explained by Theorem \ref{thm:main}: one has $\Phi'(\frac12)=104-\sum_pc_p=-1$ in case (b) and $\Phi'(\frac12)=\frac12$ in case (c).
\item Corollary \ref{cor:main}(ii) recovers \cite[Corollaries 5.10 and 5.11]{DP} when $3$ does not divide $n$.
\item Kobayashi's inequality \eqref{eq:Ko} holds also for odd $n\geq7$, and the proof of Theorem \ref{thm:main} shows that equality in \eqref{eq:Ko} implies $\sum_pc_p=2n^2-3n$ for any $n\geq7$.
\end{enumerate}
\end{remark}

\section{Some classical facts on plane curves}\label{sec:facts}

Let $C=C_1\cup\ldots\cup C_r$ be the decomposition of a reduced plane curve $C$ of degree $n$ into $r$ irreducible components. We set $d_i=\deg C_i$, we let $g_i$ be the geometric genus of $C_i$, and we put $p_a(e)=\frac{(e-1)(e-2)}2$. For a singular point $p$ of $C$ we denote by $\delta_p$ its $\delta$-invariant. We let $I_p$ be the sum of the intersection multiplicities of all pairs of distinct local branches of $C$ at $p$, and $\kappa_p$ the sum of $\mult(\beta)-1$ over all local branches $\beta$ of $C$ at $p$. Table \ref{tab:branches} describes the local branches of the ADE singularities.

\begin{table}[ht]
\centering\small
\renewcommand{\arraystretch}{1.3}
\begin{tabular}{@{}llllll@{}}
\toprule
type & local branches & intersection numbers & $\delta_p$ & $I_p$ & $\kappa_p$\\
\midrule
$A_{2k-1}$ & two smooth branches & $k$ & $k$ & $k$ & $0$\\
$A_{2k}$ & one branch of multiplicity $2$ & --- & $k$ & $0$ & $1$\\
$D_{2k}$, $k\geq2$ & smooth $L,B_1,B_2$ & $L\cdot B_i=1$, $B_1\cdot B_2=k-1$ & $k+1$ & $k+1$ & $0$\\
$D_{2k+1}$, $k\geq2$ & smooth $L$, $S$ of type $A_{2k-2}$ & $L\cdot S=2$ & $k+1$ & $2$ & $1$\\
$E_6$ & one branch of multiplicity $3$ & --- & $3$ & $0$ & $2$\\
$E_7$ & smooth $L$, a cusp $S$ tangent to $L$ & $L\cdot S=3$ & $4$ & $3$ & $1$\\
$E_8$ & one branch of multiplicity $3$ & --- & $4$ & $0$ & $2$\\
\bottomrule
\end{tabular}
\caption{Local branches of the ADE singularities.}
\label{tab:branches}
\end{table}
The following two lemmas are well-known.
\begin{lemma}\label{lem:genus}
With the above notation,
$$\sum_{p\in\Sing(C)}\delta_p=p_a(n)+r-1-\sum_{i=1}^rg_i.$$
In particular, $r\geq\sum_p\delta_p-p_a(n)+1$.
\end{lemma}

\begin{proof}
For a point $p\in\PP^2$, let $\delta(C_i,p)$ be the $\delta$-invariant of the germ $(C_i,p)$, so that $\delta(C_i,p)=0$ when $p\notin C_i$ or $C_i$ is smooth at $p$, and let $(C_i\cdot C_j)_p$ be the intersection multiplicity of $C_i$ and $C_j$ at $p$. We use the following three facts.
\begin{enumerate}[label=(\alph*)]
\item The genus formula for the irreducible plane curve $C_i$ of degree $d_i$:
$$\sum_{p\in C_i}\delta(C_i,p)=p_a(d_i)-g_i.$$
\item B\'ezout's theorem: for $1\leq i<j\leq r$ one has
$$\sum_{p\in C_i\cap C_j}(C_i\cdot C_j)_p=d_id_j.$$
\item For any point $p\in\PP^2$ one has
\begin{equation}\label{eq:deltalocal}
\delta_p=\sum_{i=1}^r\delta(C_i,p)+\sum_{1\leq i<j\leq r}(C_i\cdot C_j)_p.
\end{equation}
\end{enumerate}
To prove (c), consider first two plane curve germs $(X,p)$ and $(Y,p)$ without common irreducible components, and let $b(X,p)$ denote the number of local branches of $X$ at $p$. Milnor's formula $2\delta(X,p)=\mu(X,p)+b(X,p)-1$, see \cite[Theorem 10.5]{Mi}, together with the formula
$$\mu(X\cup Y,p)=\mu(X,p)+\mu(Y,p)+2(X\cdot Y)_p-1,$$
see \cite[Theorem 6.5.1]{Wa}, and the obvious equality $b(X\cup Y,p)=b(X,p)+b(Y,p)$, give
$$2\delta(X\cup Y,p)=\bigl(\mu(X,p)+b(X,p)-1\bigr)+\bigl(\mu(Y,p)+b(Y,p)-1\bigr)+2(X\cdot Y)_p,$$
that is
\begin{equation}\label{eq:deltaunion}
\delta(X\cup Y,p)=\delta(X,p)+\delta(Y,p)+(X\cdot Y)_p.
\end{equation}
Now let $p\in\PP^2$, and assume that $C_1,\dots,C_k$ are exactly the components of $C$ passing through $p$. All the terms in \eqref{eq:deltalocal} involving some $C_i$ with $i>k$ vanish, so \eqref{eq:deltalocal} is clear for $k\leq1$. For $k\geq2$, it follows by induction on $k$, applying \eqref{eq:deltaunion} to $X=C_1\cup\ldots\cup C_{k-1}$ and $Y=C_k$, and using the additivity $(X\cdot C_k)_p=\sum_{i<k}(C_i\cdot C_k)_p$ of intersection multiplicities.

Every singular point of a component $C_i$ and every point of $C_i\cap C_j$ with $i\neq j$ is a singular point of $C$, and both sides of \eqref{eq:deltalocal} vanish when $p$ is a smooth point of $C$. Hence summing \eqref{eq:deltalocal} over all $p\in\Sing(C)$ and using (a) and (b), we get
$$\sum_{p\in\Sing(C)}\delta_p=\sum_{i=1}^r\bigl(p_a(d_i)-g_i\bigr)+\sum_{1\leq i<j\leq r}d_id_j.$$
Finally, using $2p_a(e)=e^2-3e+2$ and $n=\sum_id_i$, a direct computation gives
$$2\sum_{i=1}^rp_a(d_i)+2\sum_{i<j}d_id_j=n^2-3n+2r=2\bigl(p_a(n)+r-1\bigr),$$
which proves the claimed formula. The last assertion follows from $g_i\geq0$.
\end{proof}

\begin{lemma}\label{lem:budget}
With the above notation,
$$\sum_{1\leq i<j\leq r}d_id_j\leq\sum_{p\in\Sing(C)}I_p,$$
with equality if and only if, at every singular point of $C$, the local branches of $C$ belong to pairwise distinct irreducible components. Moreover, the left-hand side is at least $(r-1)(n-r+1)+\binom{r-1}{2}$.
\end{lemma}

\begin{proof}
The intersection number $C_i\cdot C_j$ is the sum of the intersection multiplicities of the local branches of $C_i$ with the local branches of $C_j$, all located at singular points of $C$. The lower bound corresponds to the partition $n=(n-r+1)+\underbrace{1+\dots+1}_{r-1}$.
\end{proof}

For $n=8$ and $r=2,\dots,7$ the lower bound in Lemma \ref{lem:budget} is $7,13,18,22,25,27$, and for $n=10$ and $r=2,\dots,6$ it is $9,17,24,30,35$. Note also that lines and conics are smooth curves, hence they carry no singular local branch, and they carry at most one local branch at any given point.

For an irreducible curve $B\subset\PP^2$ of degree $e\geq2$ and geometric genus $g$, let $\kappa(B)$ be the sum of $\mult(\beta)-1$ over all local branches $\beta$ of $B$. Then the class $e^\vee$ of $B$, i.e.\ the degree of the dual curve $B^\vee$, is given by
\begin{equation}\label{eq:class}
e^\vee=2e+2g-2-\kappa(B),
\end{equation}
see \cite[Section 9.1, formula (3) on p.\ 597]{BK}. This follows from the Riemann--Hurwitz formula for the projection of the normalization of $B$ from a general point, as in \cite[Section 9.2, proof of Lemma 4]{BK}, and it is equivalent to the formula $e^\vee=e(e-1)-\sum_p(\mu_p+\mult_p-1)$ used in \cite[Section 4]{DP}, see \cite[(1.2.18)]{Di} and \cite[Section 9.1, Theorem 2(i)]{BK}. Recall that $B^{\vee\vee}=B$, see \cite[Section 6.2, p.\ 255]{BK}, and that associating to a point of $B$ its tangent line gives a birational map $B\to B^\vee$, see \cite[Section 9.1, p.\ 597]{BK}. Hence $B^\vee$ has degree $e^\vee$, geometric genus $g$ and class $e$. Applying \eqref{eq:class} to $B^\vee$, as in \cite[Section 9.1, formulae (4) and (5) on p.\ 597]{BK}, we get the following.

\begin{lemma}[Pl\"ucker]\label{lem:pluecker}
Let $B\subset\PP^2$ be an irreducible curve of degree $e\geq2$ and geometric genus $g$, and let $e^\vee$ be its class. Then $e^\vee\geq2$ and
$$0\leq\kappa(B^\vee)=2e^\vee+2g-2-e\leq\delta(B^\vee)=p_a(e^\vee)-g.$$
\end{lemma}

\begin{proof}
Apply \eqref{eq:class} to $B^\vee$, and use that a local branch of multiplicity $\mu$ has $\delta\geq\binom{\mu}{2}\geq\mu-1$.
\end{proof}

Under the birational map $B\to B^\vee$, each local branch $\beta$ of $B$, with tangent line $T$, corresponds to a local branch $\beta^\vee$ of $B^\vee$ centered at the point $T^\vee$, and the multiplicity of $\beta^\vee$ is the class of $\beta$, see \cite[Section 9.1, pp.\ 576--577]{BK}. We need the following facts about these dual branches.

\begin{lemma}\label{lem:duality}
Let $B\subset\PP^2$ be an irreducible curve of degree at least $2$ and let $B^\vee$ be its dual curve.
\begin{enumerate}[label=(\alph*)]
\item Let $\beta$ be a smooth local branch of $B$ at $p$, with tangent line $T$, and let $c=(T\cdot\beta)_p\geq2$. Then the dual branch $\beta^\vee$ of $B^\vee$ at the point $T^\vee$ has multiplicity $c-1$. In particular, $B$ has no inflection points if $\kappa(B^\vee)=0$.
\item If $t$ distinct local branches of $B$, smooth or not, have the same tangent line $L$, then $L^\vee$ is a point of $B^\vee$ with at least $t$ local branches. If two of these branches are centered at the same point $p$, the corresponding branches of $B^\vee$ have the same tangent line $p^\vee$ at $L^\vee$, so their intersection multiplicity is at least $2$.
\item If two smooth local branches of $B$ at $p$ have the same tangent line $T$, both have contact $2$ with $T$, and their intersection multiplicity is $k\geq3$, then the dual branches at $T^\vee$ are smooth and their intersection multiplicity is $k$.
\end{enumerate}
\end{lemma}

\begin{proof}
For (a), choose affine coordinates with $p=(0,0)$ and $T=\{y=0\}$, and write $\beta$ as the graph $y=f(x)$, where $f$ has order $c$. The tangent line at the point $(x,f(x))$ is $y=ux-v$, with $u=f'(x)$ and $v=xf'(x)-f(x)$. Since $B\to B^\vee$ is birational, $x\mapsto(u(x),v(x))$ is a parametrization of $\beta^\vee$, and $u$, $v$ have orders $c-1$ and $c$ in $x$. Hence $\beta^\vee$ has multiplicity $c-1$. In the language of \cite[Section 7.3, p.\ 288]{BK}, an $r$-fold inflection point of $B$ gives a branch of $B^\vee$ of multiplicity $r+1$ -- this is the correspondence used in the proof of \cite[Section 9.1, Theorem 2(ii)]{BK} on p.\ 597, and for $r=1$ it is the correspondence between simple inflection points and simple cusps \cite[Section 9.1, p.\ 583]{BK}.

For (b), note that the branch $\beta^\vee$ passes through $T^\vee$, and that its tangent line there is the line $p^\vee$ dual to the center $p$ of $\beta$. For a smooth branch $\beta$ this is the biduality statement in \cite[Section 6.2, p.\ 255]{BK}, and in general it follows by continuity, the tangent line of $\beta^\vee$ at $T^\vee$ being the limit of the lines dual to points of $\beta$ tending to $p$. Since $B\to B^\vee$ is birational, distinct branches of $B$ give distinct branches of $B^\vee$, which proves (b).

For (c), choose affine coordinates with $p=(0,0)$ and $T=\{y=0\}$, and write the two branches as graphs $y=f_j(x)=\sum_{i\geq2}a_{j,i}x^i$ with $a_{j,2}\neq0$, $j=1,2$. Since $k\geq3$, the coefficients $a_{1,i}$ and $a_{2,i}$ coincide for $i<k$, and $a_{1,k}\neq a_{2,k}$. The tangent line at the point $(x,f(x))$ is $y=ux-v$, with $u=f'(x)$ and $v=xf'(x)-f(x)$. Since $f''(0)\neq0$, the dual branch is the graph of the Legendre transform $v=F(u)$, which satisfies $F'(u)=x(u)$. Inverting $u=f'(x)$, one sees that, for $i\geq3$, the coefficient of $u^i$ in $F$ equals $-a_i/(2a_2)^i$ plus a polynomial in $a_2^{\pm1},a_3,\dots,a_{i-1}$. Hence $F_1-F_2$ has order exactly $k$ in $u$, which is our claim.
\end{proof}

Finally, if $\beta_1,\dots,\beta_s$ are the local branches of a curve at a point $q$, then $\delta_q=\sum_j\delta(\beta_j)+\sum_{i<j}(\beta_i\cdot\beta_j)_q$ and $\delta(\beta_j)\geq\mult(\beta_j)-1$. Summing over all the points of $B^\vee$ we get
\begin{equation}\label{eq:dualdelta}
\delta(B^\vee)\geq\kappa(B^\vee)+\sum_{q\in B^\vee}\ \sum_{i<j}(\beta_i\cdot\beta_j)_q .
\end{equation}

\section{Proof of Theorem \ref{thm:octics}}\label{sec:octics}

For $n=8$, the MK condition reads $\sum_p(\mu_p+1-\frac1{|G_p|})=\frac{136}{3}$, Theorem \ref{thm:main} reads $3\tau(C)-\sum_pd_p=104$, and we have $\tau(C)\leq37$.

\begin{proposition}\label{prop:octics}
The numerical types of ADE singularities satisfying these three conditions are exactly the $18$ types listed in Table \ref{tab:octics}.
\end{proposition}

\begin{proof}
By Corollary \ref{cor:main}, the singularities of type $D$ and $E$ satisfy $\sum_pd_p\leq7$. For each such choice, $\tau(C)$ is determined by Theorem \ref{thm:main}, and the singularities of type $A$ form a partition of the remaining Tjurina number satisfying one linear equation. The resulting finite list is obtained by the script in Appendix \ref{app:candidates}.
\end{proof}

\begin{table}[ht]
\centering\footnotesize
\setlength{\tabcolsep}{4pt}
\renewcommand{\arraystretch}{1.25}
\begin{tabular}{@{}llrrrrll@{}}
\toprule
 & singularities & $\tau$ & $\sum\delta_p$ & $\sum I_p$ & $\kappa$ & decompositions & argument\\
\midrule
O1 & $4A_1+4D_4+D_5+2E_6$ & 37 & 25 & 18 & 5 & none & Lemma \ref{lem:budget}\\
\midrule
N1 & $4A_1+A_2+3A_3+4D_4+D_5$ & 36 & 26 & 24 & 2 & none & Lemma \ref{lem:budget}\\
N2 & $10A_1+A_2+2D_4+2D_8$ & 36 & 27 & 26 & 1 & none & Lemma \ref{lem:budget}\\
N3 & $6A_1+4A_2+D_4+2D_5+D_8$ & 36 & 24 & 18 & 6 & $(5,1,1,1)$ & Lemma \ref{lem:pluecker}\\
N4 & $7A_1+A_2+2A_3+D_4+D_5+2D_6$ & 36 & 26 & 24 & 2 & none & Lemma \ref{lem:budget}\\
N5 & $6A_1+3A_2+A_3+A_4+D_4+D_5+E_8$ & 36 & 23 & 13 & 7 & $(6,1,1)$ & Lemma \ref{lem:pluecker}\\
N6 & $7A_1+2A_2+A_5+4D_5$ & 36 & 24 & 18 & 6 & $(5,1,1,1)$ & Lemma \ref{lem:pluecker}\\
N7 & $2A_1+7A_2+4D_5$ & 36 & 21 & 10 & 11 & $(8)$, $(7,1)$ & Lemma \ref{lem:pluecker}\\
N8 & $6A_1+3A_2+2A_3+D_5+D_6+E_7$ & 36 & 24 & 19 & 5 & $(5,1,1,1)$ & Lemma \ref{lem:pluecker}\\
N9 & $6A_1+4A_2+A_5+D_5+2E_6$ & 36 & 22 & 11 & 9 & $(7,1)$ & Lemma \ref{lem:pluecker}\\
N10 & $7A_1+4A_3+D_5+2E_6$ & 36 & 24 & 17 & 5 & none & Lemma \ref{lem:budget}\\
N11 & $A_1+9A_2+D_5+2E_6$ & 36 & 19 & 3 & 14 & $(8)$ & Lemma \ref{lem:pluecker}\\
\midrule
T1 & $13A_1+3A_2+A_{11}+D_5$ & 35 & 25 & 21 & 4 & none & Lemma \ref{lem:budget}\\
T2 & $14A_1+A_2+2A_7+D_5$ & 35 & 26 & 24 & 2 & none & Lemma \ref{lem:budget}\\
T3 & $11A_1+3A_3+2A_5+D_5$ & 35 & 26 & 25 & 1 & $(3,1,1,1,1,1)$ & Proposition \ref{prop:T3}\\
T4 & $6A_1+5A_2+3A_3+A_5+D_5$ & 35 & 23 & 17 & 6 & $(6,1,1)$, $(5,2,1)$ & Proposition \ref{prop:T4}\\
T5 & $7A_1+A_2+7A_3+D_5$ & 35 & 25 & 23 & 2 & $(4,1,1,1,1)$ & Proposition \ref{prop:T5}\\
T6 & $A_1+10A_2+3A_3+D_5$ & 35 & 20 & 9 & 11 & $(8)$, $(7,1)$ & Proposition \ref{prop:T6}\\
\bottomrule
\end{tabular}
\caption{The numerical types of MK-octics. Since $\nu(C)=37-\tau(C)$, the type O1 corresponds to a free curve, the types N1--N11 to nearly free curves, and T1--T6 to curves with $\nu(C)=2$. The column \texttt{decompositions} lists the degrees of the irreducible components allowed by Lemmas \ref{lem:genus} and \ref{lem:budget}, and \texttt{none} means no partition is consistent with the aforementioned results.}
\label{tab:octics}
\end{table}

The numbers $\sum_p\delta_p$, $\sum_pI_p$ and $\kappa=\sum_p\kappa_p$ in Table \ref{tab:octics} are obtained by adding the entries of Table \ref{tab:branches}. By Lemma \ref{lem:genus},
$$r-\sum_ig_i=\sum_p\delta_p-20,$$
and the decompositions $8=d_1+\dots+d_r$ compatible with Lemma \ref{lem:budget} are easily listed by hand. For the sanity check, the script in Appendix \ref{app:geometry} recomputes these data. Our aim not is to exclude the $18$ types and we will do it in three steps.

\subsection{Types excluded by numerics} For the types O1, N1, N2, N4, N10, T1 and T2, the number of irreducible components forced by Lemma \ref{lem:genus} is incompatible with Lemma \ref{lem:budget}. For instance, for O1 we have $\sum_p\delta_p=25$, hence $r\geq5$ and $\sum_{i<j}d_id_j\geq22$, while $\sum_pI_p=18$, a contradiction. The other cases can be treated in a similar way, namely
\begin{itemize}
\item N1, N4: $r\geq6$, so $\sum_{i<j}d_id_j\geq25>24$;
\item N2: $r\geq7$, so $\sum_{i<j}d_id_j\geq27>26$;
\item N10: $r\geq4$, so $\sum_{i<j}d_id_j\geq18>17$;
\item T1: $r\geq5$, so $\sum_{i<j}d_id_j\geq22>21$;
\item T2: $r\geq6$, so $\sum_{i<j}d_id_j\geq25>24$.
\end{itemize}

\subsection{Types excluded by the Pl\"ucker lemma} For the types N3, N5--N9 and N11, Lemmas \ref{lem:genus} and \ref{lem:budget} leave only the decompositions listed in Table \ref{tab:octics}, together with the value of $\sum_ig_i$. In each of them there is exactly one component $B$ of degree $e\geq3$, the other components are lines, so $g(B)=\sum_ig_i$, and $B$ contains all the singular local branches of $C$, so $\kappa(B)=\kappa$. The class $e^\vee$ given by \eqref{eq:class} then violates Lemma \ref{lem:pluecker}:
\begin{itemize}
\item N3 and N6: $B$ is a rational quintic with $\kappa(B)=6$, so $e^\vee=2$ and $\kappa(B^\vee)=-3$;
\item N5: $B$ is a rational sextic with $\kappa(B)=7$, so $e^\vee=3$ and $\kappa(B^\vee)=-2$;
\item N7: $B$ is either a rational octic or a septic of genus $1$, with $\kappa(B)=11$, so $e^\vee=3$ and $\kappa(B^\vee)<0$;
\item N8: $B$ is a rational quintic with $\kappa(B)=5$, so $e^\vee=3$ and $\kappa(B^\vee)=-1$;
\item N9: $B$ is a rational septic with $\kappa(B)=9$, so $e^\vee=3$ and $\kappa(B^\vee)=-3$;
\item N11: here $C$ is irreducible of genus $2$ and $\kappa(C)=14$, so $e^\vee=4$ and $C^\vee$ is a quartic of genus $2$, so that $\delta(C^\vee)=1<2=\kappa(C^\vee)$.
\end{itemize}

\subsection{The four remaining types}

\begin{proposition}\label{prop:T3}
The type T3 $=11A_1+3A_3+2A_5+D_5$ is not realized by a plane octic.
\end{proposition}

\begin{proof}
Here $r-\sum g_i=6$ and $\sum_pI_p=25$, so by Lemma \ref{lem:budget} the curve $C$ consists of a cubic $T$ and five lines, all the components are rational, and $\sum_{i<j}d_id_j=\sum_pI_p$. By the equality case of Lemma \ref{lem:budget}, at each singular point of $C$ the local branches lie on pairwise distinct components. In particular, $T$ contains the singular branch of the $D_5$ point, hence $T$ is a cuspidal cubic, and its only singular point is this cusp.

Consider now one of the two $A_5$ points. Its two branches are smooth, they have intersection multiplicity $3$, and they lie on distinct components. As two lines meet transversally, one branch lies on a line $L$ and the other one on $T$. Hence $L$ meets $T$ with multiplicity $3$ at a smooth point of $T$, i.e.\ $L$ is an inflectional tangent of $T$. Since $L\cdot T=3$, the two $A_5$ points give two distinct inflectional tangents of $T$. This is impossible, since a cuspidal cubic is projectively equivalent to $y^2z=x^3$ and has exactly one inflection point, see \cite[Section 7.3, Proposition 13]{BK}.
\end{proof}

\begin{proposition}\label{prop:T4}
The type T4 $=6A_1+5A_2+3A_3+A_5+D_5$ is not realized by a plane octic.
\end{proposition}

\begin{proof}
Here $r-\sum g_i=3$ and $\sum_pI_p=17$. Since $r\geq4$ would give us $\sum_{i<j}d_id_j\geq18$, we get $r=3$, all components rational, and the admissible decomposition is either $(5,2,1)$ or $(6,1,1)$. The six singular local branches (the five $A_2$ points and the singular branch of $D_5$) lie on the component $B$ of degree $e\geq5$, so $\kappa(B)=6$. In the case $(5,2,1)$ this gives $e^\vee=2$ and $\kappa(B^\vee)=-3$, contradicting Lemma \ref{lem:pluecker}.

Consider the case $(6,1,1)$. Then $B$ is a rational sextic of class $4$, so $B^\vee$ is a rational quartic of class $6$, with $\kappa(B^\vee)=0$ and $\delta(B^\vee)=3$. By Lemma \ref{lem:duality}(a), $B$ has no inflection points. Since $\sum_{i<j}d_id_j=13$, the intersection multiplicities between local branches lying on the same component add up to $17-13=4$. Such pairs of branches can only lie on $B$.

The $A_5$ point cannot have one branch on a line $L$ and the other on $B$ since then $L$ would be an inflectional tangent of $B$. Hence both branches of the $A_5$ point lie on $B$, and this uses $3$ of the $4$ units above. The remaining unit corresponds to one $A_1$ point with both branches on $B$. It follows that each $A_3$ point consists of a line tangent to $B$ at a smooth point of $B$.

There are three $A_3$ points and only two lines, so one of the lines, say $L$, is tangent to $B$ at two distinct points. By Lemma \ref{lem:duality}(b), $L^\vee$ is a point of $B^\vee$ with at least two branches. On the other hand, let $p$ be the $A_5$ point and $T$ the common tangent of its two branches. These branches are non-inflectional, so by Lemma \ref{lem:duality}(c) their dual branches meet at $T^\vee$ with intersection multiplicity $3$. Note that $T\neq L$, since $L$ does not pass through $p$. Now \eqref{eq:dualdelta} gives $\delta(B^\vee)\geq0+1+3=4>3$, a contradiction.
\end{proof}

\begin{proposition}\label{prop:T5}
The type T5 $=7A_1+A_2+7A_3+D_5$ is not realized by a plane octic.
\end{proposition}

\begin{proof}
Here $r-\sum g_i=5$ and $\sum_pI_p=23$, so Lemma \ref{lem:budget} gives the decomposition $(4,1,1,1,1)$, all components rational, and $\sum_{i<j}d_id_j=22$. Hence the intersection multiplicities between branches lying on the same component add up to $1$. The quartic $Q$ contains the $A_2$ point and the singular branch of $D_5$, and the remaining unit is a node of $Q$. Thus $Q$ is a rational quartic with two cusps and one node, and $Q$ is smooth at all the other points of $C$. Then $Q$ has class $4$, and $Q^\vee$ is a rational quartic with $\kappa(Q^\vee)=2$ and $\delta(Q^\vee)=3$.

No singularity in T5 has three local branches. Hence the six intersection points of the four lines are $A_1$ points not on $Q$, and together with the node of $Q$ they account for all seven $A_1$ points. It follows that
$$\sum_i L_i\cdot Q=16=7\cdot 2+2$$
comes from the seven $A_3$ points, each being a tangency of a line with $Q$, and from the $D_5$ point, where a line passes through a cusp of $Q$. Since each line meets $Q$ with multiplicity $4$, three of the four lines are bitangent to $Q$. By Lemma \ref{lem:duality}(b) this gives three distinct points of $Q^\vee$ with at least two branches, and \eqref{eq:dualdelta} yields $\delta(Q^\vee)\geq2+3=5>3$, a contradiction.
\end{proof}

\begin{proposition}\label{prop:T6}
The type T6 $=A_1+10A_2+3A_3+D_5$ is not realized by a plane octic.
\end{proposition}

\begin{proof}
Here $r=\sum g_i$ and $\sum_pI_p=9$, so $r\leq2$. Hence either $C$ is irreducible of genus $1$, or $C=S\cup L$ with $S$ a septic of genus $2$ and $L$ a line. In both cases the component of degree $\geq7$ has $\kappa=11$.

In the second case, $S$ has class $14+4-2-11=5$, and $S^\vee$ is a quintic of genus $2$ with $\kappa(S^\vee)=10+4-2-7=5>4=\delta(S^\vee)$, contradicting Lemma \ref{lem:pluecker}.

In the first case, $C$ has class $16+2-2-11=5$, and $C^\vee$ is a quintic of genus $1$, so $\delta(C^\vee)=5$. Let $T_1,T_2,T_3$ be the tangent lines at the three tacnodes of $C$. By Lemma \ref{lem:duality}(b), each $T_j^\vee$ is a point of $C^\vee$ carrying two branches with intersection multiplicity at least $2$. If the $T_j$ are distinct, \eqref{eq:dualdelta} gives $\delta(C^\vee)\geq6$. If two of them coincide, the corresponding point of $C^\vee$ has at least four branches, and \eqref{eq:dualdelta} gives $\delta(C^\vee)\geq2+2+4=8$. In both cases we get a contradiction.
\end{proof}

This completes the proof of Theorem \ref{thm:octics}.

\begin{remark}
Proposition \ref{prop:octics} shows that an MK-octic would have a singularity of type $D_5$ or $D_8$. This already recovers, in degree $8$, the results \cite[Theorem 5.7, Corollaries 5.10 and 5.11, Theorems 5.12 and 5.16]{DP}. Note also that the numerical types of free and nearly free MK-octics are excluded by the simplest arguments above, while in degree $6$ one MK-sextic is free and the other two are nearly free, see \cite[Example 6.4]{DP}.
\end{remark}

\section{Proof of Theorem \ref{thm:decics}}\label{sec:decics}

For $n=10$, the MK condition reads $\sum_p(\mu_p+1-\frac1{|G_p|})=\frac{220}{3}$, Theorem \ref{thm:main} reads $3\tau(C)-\sum_pd_p=170$, and we have $\tau(C)\leq61$.

\begin{proposition}\label{prop:decics}
There are exactly $284$ numerical types of ADE singularities satisfying these three conditions. For $\tau=57,58,59,60,61$ there are $30$, $96$, $108$, $48$ and $2$ of them, respectively. The two types with $\tau=61$ are
$$\mathrm{W}_{1}=6A_1+3D_4+D_5+4E_6+2E_7\quad\text{ and }\quad \mathrm{W}_{2}=8A_1+A_2+D_4+D_5+6E_7.$$
\end{proposition}

\begin{proof}
The proof is the same as for Proposition \ref{prop:octics}, see the script in Appendix \ref{app:candidates}.
\end{proof}

\begin{proof}[Proof of Theorem \ref{thm:decics}]
By Corollary \ref{cor:main} and Proposition \ref{prop:decics} it is enough to exclude the types F1 and F2. We use the notation of Section \ref{sec:facts}, and the data can be checked with the script in Appendix \ref{app:geometry}.

For F1 one has $\sum_p\delta_p=38$ and $\sum_pI_p=23$, hence $r-\sum g_i=3$. Lemma \ref{lem:budget} leaves only two decompositions $(8,1,1)$ and $(7,2,1)$, with all components rational. The component $B$ of degree $e\geq7$ contains all singular branches, so $\kappa(B)=4\cdot2+2+1=11$. If $e=8$, then $e^\vee=3$ and $\kappa(B^\vee)=-4$. If $e=7$, then $e^\vee=1$. Both contradict Lemma \ref{lem:pluecker}.

For F2 one has $\sum_p\delta_p=39$ and $\sum_pI_p=31$, hence $r-\sum g_i=4$, and $\kappa=8$. Lemma \ref{lem:budget} leaves three decompositions:
\begin{itemize}
\item $(7,1,1,1)$ with all components rational: here $e^\vee=4$ and $\kappa(B^\vee)=-1$, which contradicts Lemma \ref{lem:pluecker};
\item $(6,2,1,1)$ with all components rational: here $e^\vee=2$ and $\kappa(B^\vee)=-4$, again a contradiction;
\item $(6,1,1,1,1)$, where the sextic $B$ has genus $1$.
\end{itemize}
In the last case $B$ has class $4$, so $B^\vee$ is a quartic of genus $1$ with $\delta(B^\vee)=2$ and $\kappa(B^\vee)=2$. By \eqref{eq:dualdelta}, no point of $B^\vee$ can have two or more local branches.

On the other hand, $\delta(B)=p_a(6)-1=9$, and the eight singular branches on $B$ (the $A_2$ point, the singular branch of $D_5$ and the cusps of the six $E_7$ points) contribute $8$ to $\delta(B)$. If both branches of an $E_7$ point lay on $B$, they would contribute $3$ more, which is impossible. Hence for each $E_7$ point, the smooth branch lies on a line, which is therefore the tangent line to $B$ at the corresponding cusp. Each such line meets $B$ with multiplicity $3$ at this cusp, so a line can be the cuspidal tangent at no more than two of these cusps. Since there are six $E_7$ points and four lines, at least one line $L$ is tangent to $B$ at two distinct cusps. By Lemma \ref{lem:duality}(b), $L^\vee$ is a point of $B^\vee$ with at least two branches, a contradiction.
\end{proof}

\begin{remark}
A systematic, computer aided application of the arguments of Sections \ref{sec:facts} and \ref{sec:octics} to the remaining $282$ numerical types of degree $10$ excludes a large part of them, but by no means all of them. In the remaining cases the curve typically has an irreducible component of degree at least $7$, for which the counting arguments of Section 4 give no contradiction. We do not pursue this here due to lack of efficient techniques.
\end{remark}
We would like to conclude our paper by the following difficult problem.
\begin{question}
Is there a Miyaoka--Kobayashi curve of degree $10$? Is there a nearly free one?
\end{question}
\section*{Acknowledgements}
We would like to remember Wolfgang Ebeling, a friend of both authors and an exceptionally kind and warm-hearted person. It was Wolfgang who first introduced Piotr Pokora to MK-curves: after a discussion in Hannover, he handed him a copy of Ivinskis’ Diplomarbeit, a small gesture, entirely typical of his generosity, that some years later became the starting point of our interest in these curves. Wolfgang was a doctoral student of Egbert Brieskorn, and in that same spirit we have drawn many of the results used here from the wonderful book Plane Algebraic Curves by Brieskorn and Knörrer.
\section*{Funding}
Piotr Pokora is supported by the National Science Centre (Poland) Sonata Bis Grant \textbf{\path{2023/50/E/ST1/00025}}. For the purpose of Open Access, the author has applied a CC-BY public copyright license to any Author Accepted Manuscript (AAM) version arising from this submission.

\section*{AI Declaration}
All mathematical results in this paper were obtained by the authors. We used \texttt{Claude} (Anthropic) only to consolidate scattered symbolic computations in \texttt{SINGULAR} into two scripts that readers can run on their own computers.

\section*{Data Availability Statement}
The scripts supporting the computations in this paper are available in Appendix \ref{app:scripts}.

\section*{Conflict of Interest}
The authors declare that they have no conflicts of interest related to this manuscript.

\makeatletter
\enddoc@text
\let\enddoc@text\relax
\makeatother

\newpage
\appendix
\section{Local orbifold Euler numbers for ADE singularities}\label{app:eorb}
Table \ref{tab:eorb} lists the local orbifold Euler numbers of the ADE singularities for $0\leq\alpha\leq\lct_p$. For $\alpha_p\leq\alpha\leq\lct_p$, with $\alpha_p$ as in Table \ref{tab:local}, the entries agree with the formula in Lemma \ref{lem:eorb}.

\begin{table}[ht]
\centering\small
\renewcommand{\arraystretch}{1.5}
\begin{tabular}{@{}ll@{}}
\toprule
 & $\eorb(0;\CC^2,\alpha C)$\\
\midrule
$A_1$ & $(1-\alpha)^2,\quad 0\le\alpha\le1$\\ \addlinespace[3pt]
$A_{2k}$, $k\ge1$ & $\begin{cases}
1-2\alpha, & 0\le\alpha\le\frac{2k-1}{4k+2}\\[2pt]
\frac{(2k+3-(4k+2)\alpha)^2}{8(2k+1)}, & \frac{2k-1}{4k+2}<\alpha\le\frac{2k+3}{4k+2}
\end{cases}$\\ \addlinespace[3pt]
$A_{2k+1}$, $k\ge1$ & $\begin{cases}
1-2\alpha, & 0\le\alpha\le\frac{k}{2k+2}\\[2pt]
\frac{(k+2-(2k+2)\alpha)^2}{4(k+1)}, & \frac{k}{2k+2}<\alpha\le\frac{k+2}{2k+2}
\end{cases}$\\ \addlinespace[3pt]
$D_{2k}$, $k\ge2$ & $\begin{cases}
(1-2\alpha)(1-\alpha), & 0\le\alpha\le\frac{k-2}{2k-3}\\[2pt]
\frac{(k-(2k-1)\alpha)^2}{4(k-1)}, & \frac{k-2}{2k-3}<\alpha\le\frac{k}{2k-1}
\end{cases}$\\ \addlinespace[3pt]
$D_{2k+1}$, $k\ge2$ & $\begin{cases}
(1-2\alpha)(1-\alpha), & 0\le\alpha\le\frac{2k-3}{4k-4}\\[2pt]
\frac{(2k+1-4k\alpha)^2}{8(2k-1)}, & \frac{2k-3}{4k-4}<\alpha\le\frac{2k+1}{4k}
\end{cases}$\\ \addlinespace[3pt]
$E_6$ & $\begin{cases}
1-3\alpha, & 0\le\alpha\le\frac{1}{12}\\[2pt]
\frac{(7-12\alpha)^2}{48}, & \frac{1}{12}<\alpha\le\frac{7}{12}
\end{cases}$\\ \addlinespace[3pt]
$E_7$ & $\begin{cases}
1-3\alpha, & 0\le\alpha\le\frac{1}{9}\\[2pt]
\frac{(5-9\alpha)^2}{24}, & \frac{1}{9}<\alpha\le\frac{5}{9}
\end{cases}$\\ \addlinespace[3pt]
$E_8$ & $\begin{cases}
1-3\alpha, & 0\le\alpha\le\frac{2}{15}\\[2pt]
\frac{(8-15\alpha)^2}{60}, & \frac{2}{15}<\alpha\le\frac{8}{15}
\end{cases}$\\ \addlinespace[3pt]
\bottomrule
\end{tabular}
\medskip
\caption{Local orbifold Euler numbers of simple singularities
($\eorb=0$ beyond the upper bound, where the pair is not log canonical).}
\label{tab:eorb}
\end{table}

\section{\textsc{Singular} scripts}\label{app:scripts}

The following scripts were run with \texttt{SINGULAR} 4.3.2 \cite{DGPS}. 
In both scripts a singularity type is coded by an integer: $A_k$ by $k$, $D_k$ by $100+k$ and $E_k$ by $200+k$.

\subsection{The numerical types}\label{app:candidates}
This script proves Propositions \ref{prop:octics} and \ref{prop:decics}. It lists all the numerical types for $n=8$, and for $n=10$ it prints their number for each value of $\tau$ together with the two types with $\tau=61$.

\lstinputlisting{mk_candidates.sing}

The output for $n=8$ consists of the $18$ types in Table \ref{tab:octics}, listed in the same order: there are $1$, $11$ and $6$ types with $\tau=37,36,35$, respectively. For $n=10$ it gives $284$ types: there are $2$, $48$, $108$, $96$ and $30$ types with $\tau=61,60,59,58,57$, respectively, and the two types with $\tau=61$ are F1 and F2.

\subsection{Numerical data for Sections \ref{sec:octics} and \ref{sec:decics}}\label{app:geometry}
For each type in Table \ref{tab:octics} and for the types F1, F2 of Proposition \ref{prop:decics}, this script computes $\sum_p\delta_p$, $\sum_pI_p$ and $\kappa$, lists the decompositions allowed by Lemmas \ref{lem:genus} and \ref{lem:budget}, and applies the Pl\"ucker test of Lemma \ref{lem:pluecker} when there is a single component of degree at least $3$. The test fails in all cases except T3, T4 with $(6,1,1)$, T5, T6 with $(8)$ and F2 with $(6,1,1,1,1)$. These are exactly the cases treated by hand in Propositions \ref{prop:T3}--\ref{prop:T6} and in the proof of Theorem \ref{thm:decics}.

\lstinputlisting{mk_geometry.sing}

\end{document}